\documentclass[11pt]{amsart}

\title[A'Campo's divide and Turaev's shadow]{Milnor fibration, A'Campo's divide and \\ Turaev's shadow}

\author{Masaharu Ishikawa}
\address{Department of Mathematics, Hiyoshi Campus, Keio University, 4-1-1 Hiyoshi, Kohoku-ku, Yokohama-shi, Kanagawa, 223-8521, Japan}
\email{ishikawa@keio.jp}

\author{Hironobu Naoe}
\address{Department of Mathematics, Chuo University, 1-13-27 Kasuga, Bunkyo-ku, Tokyo, 112-8551, Japan}
\email{naoe@math.chuo-u.ac.jp}

\usepackage{amsfonts,amsmath,amssymb,amscd}
\usepackage{amsthm}
\usepackage{latexsym}
\usepackage{graphicx}
\usepackage{psfrag}
\usepackage{color}
\usepackage{ulem}

\theoremstyle{plain}
\newtheorem*{theorem*}{Theorem}
\newtheorem*{lemma*} {Lemma}
\newtheorem*{corollary*} {Corollary}
\newtheorem*{proposition*}{Proposition}
\newtheorem*{conjecture*}{Conjecture}
\newtheorem{theorem}{Theorem}[section]
\newtheorem{lemma}[theorem]{Lemma}
\newtheorem{corollary}[theorem]{Corollary}

\theoremstyle{remark}

\newtheorem*{remark}{Remark}

\newtheorem*{definition}{Definition}

\newtheorem*{example}{Example}

\theoremstyle{definition}

\newtheoremstyle{citing}
  {}
  {}
  {\itshape}
  {}
  {\bfseries}
  {.}
  {.5em}
  {\thmnote{#3}}

\theoremstyle{citing}

\newcommand{\Integer}{\mathbb{Z}}
\newcommand{\Real}{\mathbb{R}}
\newcommand{\Complex}{\mathbb{C}}

\newcommand{\Int}{\mathrm{Int}}

\newcommand{\Nbd}{\mathrm{Nbd}}

\newcommand{\bd}{\partial}

\newcommand{\Sing}{\mathrm{Sing}}
\newcommand{\gl}{\mathfrak{gl}}
\newcommand{\oP}{{\overrightarrow{P}}}
\newcommand{\oQ}{{\overrightarrow{Q}}}
\newcommand{\oPP}{{\overrightarrow{P}_2}}

\definecolor{darkred}{rgb}{.80,.0,.0}

\makeatletter
\@addtoreset{equation}{section}

\makeatother

\begin{document}

\maketitle

\begin{abstract}
We give a method for constructing a shadowed polyhedron from a divide.
The 4-manifold reconstructed from a shadowed polyhedron admits the structure of a Lefschetz fibration if it satisfies a certain property, which 
is formulated as an {\it LF-structure} on a shadowed polyhedron.
We will show that the shadowed polyhedron constructed from a divide satisfies this property
and the Lefschetz fibration of this polyhedron is isomorphic to the Lefschetz fibration of the divide.
Furthermore, applying the same technique to certain free divides we will show that the links of  those free divides are fibered with positive monodromy.
\end{abstract}

\section{Introduction}

A divide is the image of a generic and relative immersion of a finite number of intervals and circles into the unit disk, which was introduced by N.\,A'Campo in~\cite{AC99,AC98} as a generalization of real morsified curves of complex plane curve singularities~\cite{AC75a, AC75b, GZ74a, GZ74b, GZ77}.
A link in $S^3$ is defined from a divide and this link is fibered if the divide is connected.
Furthermore, if a divide is a real morsified curve of a complex plane curve singularity then its link is isotopic to the link of the singularity and the fibration is isomorphic to its Milnor fibration.
In~\cite{Ishi04}, the first author reformulated the fibration structure of a divide in terms of a Lefschetz fibration and generalized the definition of divides in the unit disk to those in compact orientable surfaces. In this generalized setting, the unit disk bundle in the cotangent bundle\footnote{
The Lefschetz fibration of a divide is constructed in the cotangent bundle of $\Sigma_{g,n}$
rather than the tangent bundle, though it was not carefully observed in~\cite{Ishi04}.
In this paper, according to the original paper of A'Campo, we call an element in the bundle a {\it tangent vector}  though it is a cotangent vector in actuality.
}
over a compact orientable surface is the total space of the Lefschetz fibration.
In the case of Milnor fibration, the total space corresponds to the Milnor ball,
a regular fiber corresponds to a Milnor fiber
and its boundary corresponds to the link of the singularity.

\begin{figure}[htbp]
\includegraphics[scale=0.6]{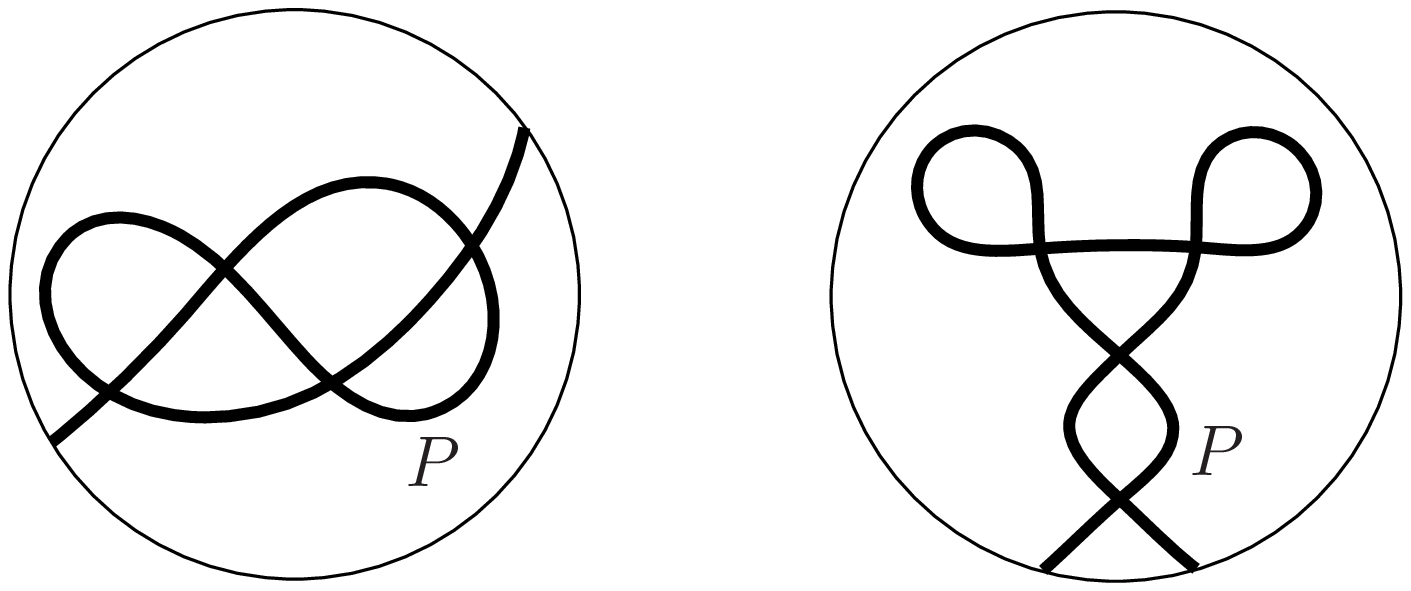}
\caption{Divides in the unit disk: The left is a divide of the $(3,5)$-torus knot, which is a real morsified curve of the singularity of $f(x,y)=x^3-y^5$. The right one does not come from a singularity. The link of this divide is $10_{139}$.}
\label{fig:divides}
\end{figure}

One may wonder how a Milnor fiber is embedded in the Milnor ball.
It is possible to guess the position sensuously, but it is not easy to describe it concretely.
In this paper, we use Turaev's shadow~\cite{Tur92, Tur94} to explain how the fiber surface is embedded.
Let $W$ be a compact, oriented, smooth $4$-manifold with boundary $M$ and $L$ be a link in $M$.
A shadow $X$ of $(W,L)$ is a simple polyhedron obtained from $W$ by collapsing
with keeping the link $L$.
Conversely, if the shadow $X$ is given then there exists an assignment $\gl$ of half integers to regions of $X$, called a {\it gleam}, such that the pair $(W,L)$ is recovered from $(X, \gl)$ uniquely.
This method is called {\it Turaev's reconstruction}.

Let $P$ be an admissible divide on a compact orientable surface $\Sigma_{g,n}$ of genus $g$ and with $n$ boundary components.
The admissibility condition is needed to have the structure of a Lefschetz fibration in the total space, see Section~2 for the definition of an admissible divide.
Now we double the curve of $P$ as follows (see Figure~\ref{fig:doubling}):
\begin{itemize}
\item[1.] double the curve of $P$;
\item[2.] for each endpoint of $P$, close the corresponding two endpoints of the doubled curve by a small half circle;
\item[3.] for each edge of $P$ that is not adjacent to an endpoint, add a crossing between the  two edges of the doubled curve parallel to the edge.
\end{itemize}
The obtained doubled curve is a divide and we denote it by $P_2$.
Note that this doubling method is similar to the one introduced in~\cite{GI02a}, but here
we add a crossing in the middle of each edge of $P$.
Let $(X_P, \gl_P)$ be a shadowed polyhedron obtained from $\Sigma_{g,n}$ by attaching 
an annulus along 
one of the boundary components to each immersed circle of 
$P_2$ and assigning a gleam $\gl_P$ to internal regions as follows:
\begin{itemize}
\item[4.] assign $\frac{1}{2}$ to each of the two triangular regions corresponding to an edge of $P$ not adjacent to an endpoint;
\item[5.] assign $0$ to the bigon corresponding to an endpoint of $P$;
\item[6.] assign $-1$ to the remaining internal regions.
\end{itemize}
We call  $(X_P, \gl_P)$ a {\it shadowed polyhedron of $P$}.

\begin{figure}[htbp]
\includegraphics[scale=0.7]{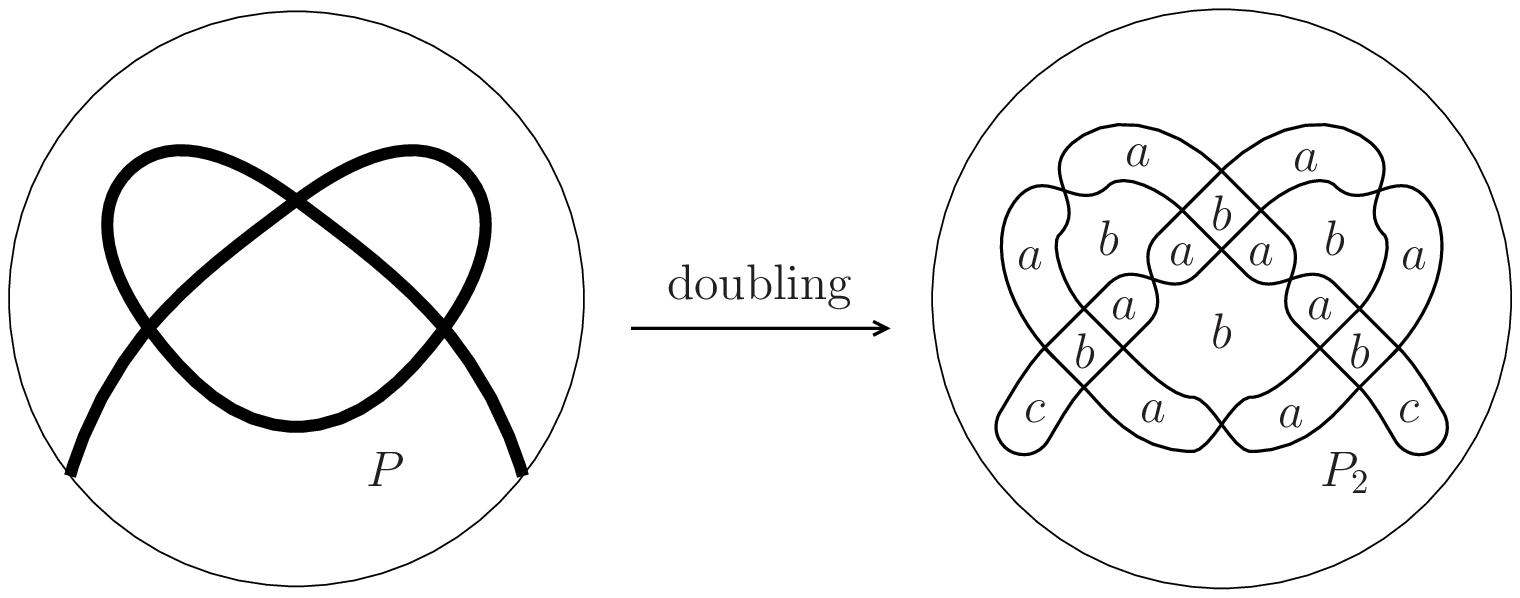}
\caption{The figure on the right represents the shadowed polyhedron of the divide on the left. The polyhedron is obtained by attaching an annulus along the doubled curve. The gleams of the internal regions labeled $a$, $b$, $c$ are $\frac{1}{2}$, $-1$, $0$, respectively.}
\label{fig:doubling}
\end{figure}

In a specific case, a shadowed polyhedron has the structure of a Lefschetz fibration canonically.

\begin{definition}
Let $(X,\gl)$ be a shadowed polyhedron.
If there exist a sub-polyhedron $X'$ of $X$ and ordered disk regions $D_1,\ldots,D_n\subset X'$ such that
\begin{itemize}
\item[(i)] $X$ collapses onto $X'$, so that the gleam $\gl'$ of $X'$ is induced from $\gl$,
\item[(ii)] 
$\partial D_i$ and $\partial D_j$ ($i\ne j$) intersect only at true vertices, 
\item[(iii)] $X'\setminus\bigl(D_1\cup\cdots\cup D_n \bigr)$ is homeomorphic to a compact, orientable surface $\Sigma$,
\item[(iv)] there exists an orientation on $\Sigma$ such that the gleam given as the sums of local contributions around crossing points of $\partial D_i$ and $\partial D_j$ with $i<j$ shown in Figure~\ref{fig:local_contribute2} coincides with the gleam $\gl'$ on each internal region of $X'$ contained in $\Sigma$, and
\item[(v)] for each $i=1,\ldots, n$, the gleam $\gl'$ of the region $D_i$ is $-1$, 
\end{itemize}
then the tuple $\mathcal{X}=(X';D_1,\ldots,D_n)$ is called an {\it LF-structure} on $(X,\gl)$. 
\end{definition}

\begin{figure}[htbp]
\includegraphics[scale=0.75]{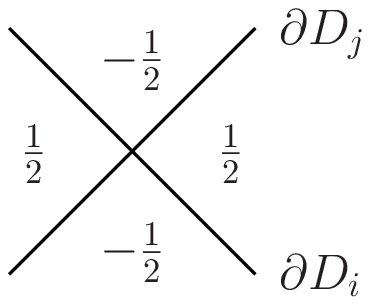}
\caption{Local contribution around a crossing of $\partial D_i$ and $\partial D_j$ with $i<j$.}
\label{fig:local_contribute2}
\end{figure}

If a shadowed polyhedron $(X,\gl)$ has an LF-structure then the corresponding $4$-manifold $W(X,\gl)$ has the structure of a Lefschetz fibration whose regular fiber is $\Sigma$ and singular points correspond to the internal regions $D_1,\ldots,D_n$.
Conversely, a Lefschetz fibration is constructed from $\Sigma\times D^2$, where $D^2$ is a $2$-disk, by attaching $2$-handles along disjoint simple closed curves on fibers over $\partial D^2$ with surface framing minus $1$.
Hence the polyhedron $X$ obtained from $\Sigma$ by attaching the cores of the $2$-handles with gleam $-1$ is a shadow of the total space of the Lefschetz fibration and we may assign a suitable gleam $\gl$ to the internal regions on $\Sigma$ such that $(X,\gl)$ has an LF-structure.
It will be proved in Lemma~\ref{LF-property} that the shadowed polyhedron $(X_P, \gl_P)$ of an admissible divide $P$  has an LF-structure $\mathcal{X}_P$.

The main theorem of this paper is the following.

\begin{theorem}\label{thm1}
Let $(X_P, \gl_P)$ be the shadowed polyhedron of an admissible divide $P$.
Then the Lefschetz fibration of $\mathcal{X}_P$ on $(X_P, \gl_P)$ coincides with that of $P$.
\end{theorem}

The fiber surface of $P$ is the surface embedded in $X_P$ and bounded by
$\partial X_P\setminus \partial{\Sigma_{g,n}}$.
An advantage of $X_P$ is that we can see both of the fiber surface and the surface $\Sigma_{g,n}$ in the polyhedron. In the case of the Milnor fibration, the latter corresponds to the real plane in the Milnor ball.
By recovering the total space according to the gleam $\gl_P$, we may understand precisely how the fiber surface is embedded in the Milnor ball with respect to this real plane.

The detection of the structure of a Lefschetz fibration by an LF-structure can also be used for certain free divides.
A free divide is a divide whose endpoints are not necessary on the boundary of the unit disk. 
It was introduced by Gibson and the first author in~\cite{GI02b}, where they defined links associated with free divides and studied their properties.
In a special case, we can show that a free divide has an LF-structure and has a structure of a Lefschetz fibration. 
An endpoint of a free divide is called a {\it free endpoint} if the region adjacent to the endpoint is bounded by the curve of the free divide.

\begin{theorem}\label{thm2}
Let $Q$ be a free divide in the unit disk $D$ consisting of one immersed interval and with one free endpoint. Starting at the free endpoint along $Q$, let $c$ be the double point of $Q$ met first.
Assume either
\begin{itemize}
\item[(1)] $c$ is on the boundary of the region adjacent to the boundary of $D$, or 
\item[(2)] the immersed arc on $Q$ connecting $c$ and the non-free endpoint passes exactly one double point.
\end{itemize}
\noindent
Then the link of $Q$ is fibered and the fibration is obtained as the boundary of a Lefschetz fibration. In particular, its monodromy is positive.
\end{theorem}

Here a monodromy is said to be {\it positive} if it is represented as a product of right-handed Dehn twists.

Among the free divides listed in~\cite{GI02b}, for example, the links of $3_{17}$ and $3_{18}$ are not fibered. Actually, they do not satisfy the assumption in Theorem~\ref{thm2}.
In the list, there are two knots, up to $10$ crossings, that are represented by free divides with one free endpoint, satisfy the conditions in Theorem~\ref{thm2} and neither closed positive braids nor links of divides, which are $10_{154}$ and $10_{161}$.

\begin{corollary}\label{cor3}
The fibered knots $10_{154}$ and $10_{161}$ are obtained as the boundaries of Lefschetz fibrations. In particular, their monodromies are positive.
\end{corollary}

As we mentioned, the link of a divide has the structure of a Lefschetz fibration. A closed positive braid also has this property, which follows from the ``anthology'' in~\cite{NR87} and the fact that it can be constructed by successive Murasugi-sum's of torus links of type $(2,k)$.
One can see that the fiber surfaces of  $10_{154}$ and $10_{161}$ are obtained by plumbing positive Hopf bands. Hence Corollary~\ref{cor3} also follows from ``anthology'' and plumbings.

The relation between divides and shadows was suggested by Professor Norbert A'Campo when the first author was a student in Universit\"{a}t Basel though he, the first author, could not catch the point at that time. 
The authors would like to thank him for introducing them to these two interesting topics.
They are also grateful to Burak \"{O}zba\u{g}ci for telling us the orientation issue of the bundle in~\cite{Ishi04}, Mikami Hirasawa for telling us about Hopf plumbings of $10_{154}$ and $10_{161}$, and Seiichi Kamada and Yuya Koda for precious comments.
The first author is supported by the Grant-in-Aid for Scientific Research (C), JSPS KAKENHI Grant Number 16K05140. 
The second author is supported by the Grant-in-Aid for Research Activity start-up, JSPS KAKENHI Grant Number 18H05827.
This work is supported by the Grant-in-Aid for Scientific Research (S), JSPS KAKENHI Grant Number 17H06128.

\section{Preliminaries}

In this paper, $\partial A$ means the boundary of a topological space $A$ and,
for topological spaces $A$ and $B$ with $A\subset B$, 
$\Nbd(A;B)$ means a small compact neighborhood of $A$ in $B$.

\subsection{A'Campo's divide}\label{sec:2.1}

Let $\Sigma_{g,n}$ be a compact, orientable, smooth surface of genus $g$ and
with $n$ boundary components with an arbitrary Riemannian metric, where $g,n\geq 0$.

\begin{definition}\label{dfn21}
A {\it divide} $P$ in $\Sigma_{g,n}$ is the image of a generic and relative immersion of 
a finite number of copies of the unit interval or the unit circle into
$\Sigma_{g,n}$. The generic condition is the following:
\begin{itemize}
   \item the image has neither self-tangent points nor triple points;
   \item 
an immersed interval intersects $\bd\Sigma_{g,n}$ at the endpoints transversely;
   \item an immersed circle does not intersect $\bd\Sigma_{g,n}$.
\end{itemize}
\end{definition}

If $\Sigma_{g,n}$ is closed then we set $N(\Sigma_{g,n})=\Nbd(\Sigma_{g,n};T(\Sigma_{g,n}))$,
where $T(\Sigma_{g,n})$ is the total space of the tangent bundle of $\Sigma_{g,n}$.

If $\Sigma_{g,n}$ has boundary, we define $N(\Sigma_{g,n})$ as follows:
Set $A=\Nbd(\partial \Sigma_{g,n};\Sigma_{g,n})$
and $B=\Sigma_{g,n}\setminus A$.
First thicken $B$ in $T(\Sigma_{g,n})$ as
\[
   \hat B:=\{(x,u)\in T(\Sigma_{g,n})\mid x\in B, u\in T_x(\Sigma_{g,n}), \|u\|\leq\varepsilon\},
\]
where $T_x(\Sigma_{g,n})$ is the tangent space to $\Sigma_{g,n}$ at $x$ and $\varepsilon>0$. 
Next, set 
$\alpha:=\bd A\setminus\bd\Sigma_{g,n}$, which is the boundary of the annuli $A$ not contained in $\bd\Sigma_{g,n}$,
and choose a compact tubular neighborhood $\Nbd(\alpha;T(\Sigma_{g,n}))$ of $\alpha$
suitably such that the boundary of
$\Nbd(\alpha;T(\Sigma_{g,n}))\cup \hat B$ becomes a smooth $3$-manifold.
Then we define $N(\Sigma_{g,n})= \Nbd(\alpha;T(\Sigma_{g,n}))\cup \hat B$.
Note that $\partial N(\Sigma_{g,n})$ is diffeomorphic to a connected sum of $2g+n-1$ copies of $S^2\times S^1$ if $n\geq 1$. In particular, it is $S^3$ if $g=0$ and $n=1$.

\begin{definition}\label{dfn41}
The {\it link} of a divide $P$ in $\Sigma_{g,n}$ is the set $L(P)$
defined by
\[
    L(P):=\{(x,u)\in\partial N(\Sigma_{g,n})\mid x\in P,\;\, u\in T_x(P)\},
\]
where $T_x(P)$ is the set of tangent vectors to $P$ at $x$.
\end{definition}

To be precise, as mentioned in the footnote of the first page, we need to reverse the orientation of $N(\Sigma_{g,n})$, or equivalently, need to replace the tangent bundle of $\Sigma_{g,n}$ in the above construction with the cotangent bundle.

Each connected component of $\Sigma_{g,n}\setminus P$ is called a {\it region} of $P$.
If a region of $P$ is bounded by $P$ then it is called
an {\it inside region}, and otherwise it is called an {\it outside region}.

\begin{definition}\label{dfn22}
A divide $P$ in $\Sigma_{g,n}$ is {\it admissible} if it satisfies the following:
\begin{itemize}
   \item $P$ is connected;
   \item each inside region of $P$ is simply connected;
   \item each outside region of $P$ is either simply connected or
an annulus such that one boundary component is a component of $\bd\Sigma_{g,n}$
and the other is contained in $P$;
   \item each component of $\bd\Sigma_{g,n}$ either does not intersect $P$ or intersects $P$ at an even number of points transversely;
   \item each circle component of $P$ intersects the other components
of $P$ at an even number of points transversely.
\end{itemize}
\end{definition}

In the case where $g=0$ and $n=1$ (i.e., $\Sigma_{g,n}$ is a disk),
a divide $P$ is admissible if and only if it is connected.
In~\cite{AC98}, A'Campo proved that if $P$ in $\Sigma_{0,1}$ is connected then
$L(P)$ is fibered with positive monodromy.
The admissibility condition was introduced in~\cite{Ishi04} to inherit this fiberedness property
to the general setting.

\begin{theorem}[Ishikawa~\cite{Ishi04}]
If a divide $P$ is admissible then $L(P)$ is a fibered link in $\partial N(\Sigma_{g,n})$ with positive monodromy.
\end{theorem}

The fibration of the fibered link $L(P)$ is obtained as the boundary of a Lefschetz fibration.
We here explain how the Lefschetz fibration is obtained briefly. See~\cite{Ishi04} more precise explanation.

Let $f_P:\Sigma_{g,n}\to\Real$ be a Morse function on $\Sigma_{g,n}$ such that $f_P^{-1}(0)=P$ and each inside region of $P$ has exactly one singular point of $f_P$.
The existence of such a Morse function is guaranteed by the admissibility condition.
Define a map $F_P:T(\Sigma_{g,n})\to \Complex$ by
\[
   F_P(x,u)=f_P(x)+idf_P(x)(u)-\frac{1}{2}\chi(x)H_{f_P}(x)(u,u),
\]
where $i=\sqrt{-1}$, $x\in \Sigma_{g,n}$, $u\in T_x(\Sigma_{g,n})$, $H_{f_P}$ is the Hessian of $f_P$ and $\chi(x)$ is a bump function which is $0$ outside small neighborhoods of double points of $P$ and $1$ on smaller neighborhoods.

Now let $D_\eta$ be 
the disk in $\Complex$ centered at the origin with sufficiently small radius $\eta>0$ and restrict $F_P$ to $F_P^{-1}(D_\eta)\cap N(\Sigma_{g,n})$.
This is a Lefschetz fibration with only one singular fiber. Note that the number of Morse singularities on the singular fiber is same as the number of double points of $P$.
Let $R_1,\ldots,R_m$ be the inside regions of $P$ and $R_i'$ be the closure of 
$R_i\setminus F^{-1}_P(D_\eta)$ for $i=1,\ldots,m$.
The total space $N(\Sigma_{g,n})$ can be recovered, up to isotopy, from  $F_P^{-1}(D_\eta)\cap N(\Sigma_{g,n})$ by attaching $R'_1\times [0,1], \ldots, R'_m\times [0,1]$ along
the simple closed curves 
$\partial (F_P^{-1}(D_\eta))\cap (R_1'\cup\cdots\cup R_m')$.
We can check directly that the framings of these attachings are those of the fiber surface of the Lefschetz fibration minus $1$. Thus the Lefschetz fibration on  $F_P^{-1}(D_\eta)\cap N(\Sigma_{g,n})$ extends to $R'_1\times [0,1], \ldots, R'_m\times [0,1]$ after these attachings.
We call this fibration the {\it Lefschetz fibration of the admissible divide $P$}.

We here note known studies related to divides.
A divide is defined first in the unit disk by A'Campo~\cite{AC99,AC98}.
In this case, the link of a divide is defined in the unit sphere $S^3$.
He then proved that if a divide is connected then the link is fibered, if the divide is a real morsified curve of a complex plane curve singularity then its fibration is isomorphic to the Milnor fibration, and if a divide consists of only immersed intervals then the unknotting number of the link is equal to the number of double points. Furthermore, in~\cite{AC98b}, he proved that there are many links of divides that are hyperbolic. The link-types of the links of divides had been studied by Couture-Perron~\cite{CP00}, Hirasawa~\cite{Hir02}, Goda-Hirasawa-Yamada~\cite{GHY02} and Kawamura~\cite{Kaw02}. Some mysterious relation between divides and exceptional surgeries had been studied by Yamada~\cite{Yam06, Yam09}.
Recently, Fomin-Pylyavskyy-Shustin studied real morsified curves and divides using quivers~\cite{FPS17} and \"{O}zba\u{g}ci used divides on compact surfaces for constructing specific Lefschetz fibrations and open book decompositions~\cite{Ozb18}.

\subsection{Turaev's shadow}
If each point of a compact space $X$ has a neighborhood homeomorphic to one of (i)-(v) 
in Figure~\ref{fig:local_model}, then $X$ is called a \textit{simple polyhedron}. 
The set of points of type (ii), (iii) and (v) is called the \textit{singular set} of $X$ 
and denoted by $\Sing(X)$. 
A point of type (iii) is a \textit{true vertex}, 
and each connected component of $\Sing(X)$ 
with true vertices removed is called a \textit{triple line}. 
Each connected component of $X\setminus\Sing(X)$ is called a \textit{region}. 
Hence a region consists of points of type (i) or (iv). 
A region is called an \textit{internal region} if it contains no points of type (iv), 
and a \textit{boundary region} otherwise. 
The \textit{boundary} of $X$, denoted by $\partial X$, 
is defined as the set of points of type (iv) and (v). 

\begin{figure}[htbp]
\includegraphics[scale=0.6]{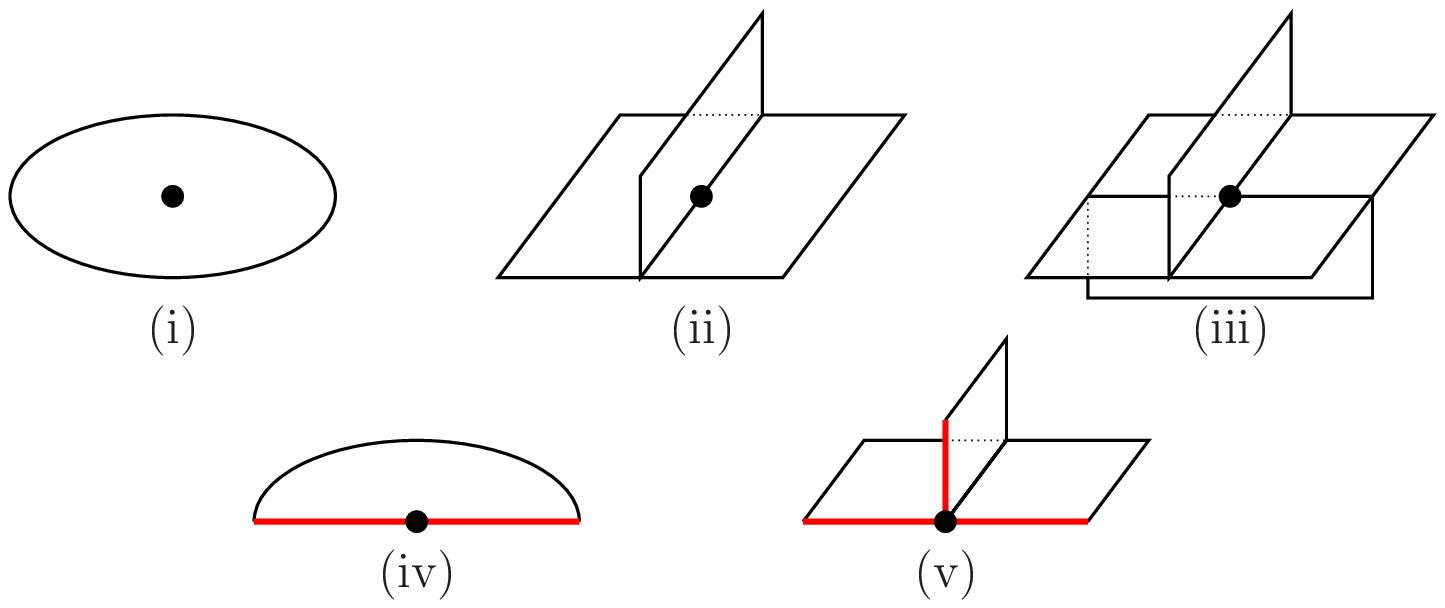}
\caption{The local models of a simple polyhedron.}
\label{fig:local_model}
\end{figure}

\begin{definition}
Let $W$ be a $4$-manifold with boundary 
and $X\subset W$ be a simple polyhedron that is proper and locally flat in $W$. 
If $W$ collapses onto $X$ after giving some triangulation to $(W,X)$, 
then the polyhedron $X$ is called a \textit{shadow} of $W$. 
\end{definition}

Here $X$ is said to be proper in $W$ if $X\cap\partial W=\partial X$ and
locally flat in $W$ if there is a local chart $(U,\varphi)$ around each point of $X$ 
such that $\varphi(U\cap X)$ is contained in $\Real^3\subset\Real^4=\varphi(U)$. 
It is easy to see that any handlebody consisting of $0$-, $1$- and $2$-handles 
admits a shadow~\cite{Tur94, Cos05}. 

For any simple polyhedron $X$, 
one can define the \textit{$\Integer_2$-gleam} on each internal region. 
Let $R$ be an internal region, and $i:F\to X$ be a continuous map 
extended from the inclusion of $R$, 
where $F$ is a compact surface whose interior is homeomorphic to $R$. 
Note that the restriction $i|_{\Int(F)}$ coincides with the inclusion of $R$, 
and that $i(\partial F)\subset\Sing(X)$. 
We now see that there exists a local homeomorphism $\tilde{i}:\tilde{F}\to X$ 
such that its image is a neighborhood of $i(F)$ in $X$, 
where $\tilde{F}$ is a simple polyhedron obtained from $F$ by 
attaching an annulus or a M\"obius strip along its core circle to each boundary component of $F$. 
Note that $\tilde{F}$ is determined up to homeomorphism from the topology of $X$. 
Here the $\Integer_2$-gleam $\mathfrak{gl}_2(R)$ of $R$ is defined to be $0$ 
if the number of the attached M\"obius strips is even, and $1$ otherwise. 

\begin{definition}
A \textit{gleam} on a simple polyhedron $X$ is a coloring for all the internal regions of $X$ 
suth that each value $\mathfrak{gl}(R)$ on an internal region $R$ satisfies $\mathfrak{gl}(R)-\frac{1}{2}\mathfrak{gl}_2(R)\in\Integer$. 
We call a pair $(X,\mathfrak{gl})$ a \textit{shadowed polyhedron}. 
\end{definition}
\begin{theorem}
[Turaev~\cite{Tur94}]
\begin{enumerate}
\item 
There exists a canonical way to construct a $4$-manifold $W$
from a given shadowed polyhedron $(X,\mathfrak{gl})$ such that $X$ is a shadow of $W$. 
This construction provides a smooth structure on $W$ uniquely. 
\item
For a $4$-manifold $W$ admitting a shadow $X$, 
there exists a gleam $\mathfrak{gl}$ on $X$ 
such that $W$ is diffeomorphic to the $4$-manifold constructed from the shadowed polyhedron $(X,\mathfrak{gl})$
according to the way of  {\rm (1)}. 
\end{enumerate}
\end{theorem}

The construction in (1) is called Turaev's reconstruction. 
A gleam plays a role as a framing coefficient to attach a $2$-handle 
in the original proof of Turaev's reconstruction. 
It is also regarded as a generalized Euler number of an embedded surface in a $4$-manifold. 
In the case where a $4$-manifold is a $D^2$-bundle over a surface $F$, 
the $4$-manifold has a shadow $F$ and the Euler number of $F$ coincides with 
the gleam coming from the above theorem. 

As we mentioned in Introduction, if a shadowed polyhedron $(X,\gl)$ has an LF-structure then $W(X,\gl)$ admits the structure of a Lefschetz fibration.

As far as we know, the first paper that relates shadows and singularity theory is the paper of Costantino and Thurston~\cite{CT08}, where the Stein factorization of a stable map from a $3$-manifold to $\Real^2$ is regarded as a shadow after a small perturbation if necessary.
In~\cite{IK17}, Koda and the first author focused on this relation and studied a relationship between the minimal number of true vertices of shadows and the minimal number of specific singular fibers of stable maps. 
Shadows are used in the study of quantum invariants by various authors,
see for instance \cite{Tur92, Tur94, Bur97, Shu97, Gou98}. 
In particular Carrega and Martelli constructed a shadow containing a given ribbon surface in $D^4$ and studied the Jones polynomial of a ribbon link \cite{CM17}. 
Concerning studies of $4$-manifolds, 
Costantino studied almost complex structures and Stein structures of $4$-manifolds with shadow representatives~\cite{Cos06, Cos08}, and the second author studied shadow representatives of corks, which yield exotic pairs of $4$-manifolds~\cite{Nao17_ojm, Nao_preprint}.
A study of classification of $4$-manifolds according to the numbers of vertices of shadows is now in progress,
see~\cite{Cos06b, Mar11, Nao17,  KMN18_preprint, KN19_preprint}.

\section{From divide to shadow}\label{sec:3}

We first show how to get a shadowed polyhedron of an oriented divide and then prove Theorem~\ref{thm1}.
An oriented divide was introduced in~\cite{GI02a} by Gibson and the first author to determine the link-type of the link of a divide.

\begin{definition}
An {\it oriented divide} $\oP$ in $\Sigma_{g,n}$ 
is the image of a generic immersion of oriented circles into $\Sigma_{g,n}$. 
\end{definition}

\begin{definition}
The {\it link} of an oriented divide $\oP$ in $\Sigma_{g,n}$ is the set  
$L(\oP)$
defined by
\[
    L(\oP):=\{(x,u)\in\partial N(\Sigma_{g,n})\mid x\in \oP,\;\, u\in T_x^+(\oP)\},
\]
where $T_x^+(P)$ is the set of tangent vectors to $\oP$ at $x$ in the same direction as $\oP$.
\end{definition}

Note that for any oriented link in $\partial N(S_{g,n})$ there exists an oriented divide $\oP$ such that the link is isotopic to $L(\oP)$, see~\cite{GI02a}.

Let $q_1,\ldots,q_\ell$ be the images of circles of $\oP$.
For each point $x\in \oP$, let $I(x)$ denote the segment in $N(\Sigma_{g,n})$
consisting of the point $x$ and the points corresponding to the tangent vectors to $\oP$ at $x$ in the same direction as $\oP$.
For each $i=1,\ldots,\ell$, the union $\bigcup_{x\in q_i}I(x)$ is an annulus one of whose boundary component lies on $\Sigma_{g,n}$ and the other lies on $\partial N(\Sigma_{g,n})$.
We denote it by $R(q_i)$.
Then the union of $\Sigma_{g,n}$ and $R(q_i)$, $i=1,\ldots,\ell$, constitutes a simple polyhedron embedded in $N(\Sigma_{g,n})$. We denote this polyhedron by $X_{\oP}$.
Note that the internal regions of $X_{\oP}$ correspond to
the inside regions of $\oP$ on $\Sigma_{g,n}$.

Next we assign a gleam to $X_{\oP}$.
For each inside region $R$ of $\oP$, we define a local contribution to
the gleam at each double point of $\oP$ on $\partial R$ as shown in Figure~\ref{fig:local_contr}. 
In the figure, the curve is  a part of $\oP$ along which the annuli $R(q_i)$'s are attached.
The gleam of $R$ is given as the sum of the local contributions minus the Euler characteristic of the region. We denote this gleam by $\gl_\oP$.

\begin{figure}[htbp]
\includegraphics[scale=0.75]{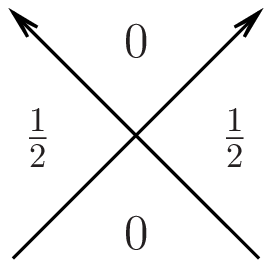}
\caption{Local contribution around a double point of an oriented divide. }
\label{fig:local_contr}
\end{figure}

\begin{lemma}
The pair $(N(\Sigma_{g,n}), X_\oP)$ is obtained from the shadowed polyhedron
$(X_\oP, \gl_\oP)$ by Turaev's reconstruction.
\end{lemma}

\begin{proof}
The vector field on the left in Figure~\ref{fig:shadow_framing} represents the annulus regions of $X_\oP$ attached along $\oP$. 
We may isotope these annulus regions in $N(\Sigma_{g,n})$ relatively to the position represented by the vector field on the right. 
We denote it by $v_{\partial R}$. 
The two vectors at the crossing are both horizontal, which means that the corresponding polyhedron is locally embedded in $\Real^3$. 
Hence we can regard $X_\oP$ as a shadow of $N(\Sigma_{g,n})$. 
Let $R$ be an internal region of $X_\oP$.
It is sufficient to check that the gleam of $R$ determined from the above embedding of $X_\oP$ into $N(\Sigma_{g,n})$ coincides with $\gl_\oP(R)$. 
Let $H=\Nbd(\partial R;N(\Sigma_{g,n}))$ and $\overline{R}=R\setminus\Int(H)$. 
Note that $\Int(\overline{R})$ is homeomorphic to $R$.
There exists an annulus or a M\"obius strip, denoted by $A$, along $\partial \overline{R}$ in $\partial H$ according to $v_{\partial R}$ as shown in Figure~\ref{fig:thickened_vertex}. 
Let $v_{\partial \overline{R}}$ be a non-zero vector field along $\partial \overline{R}$ consisting of vectors tangent to $\partial \overline{R}$
and $B$ be the annulus along $\partial \overline{R}$ in $\partial H$ that is associated with $v_{\partial \overline{R}}$. 
After suitable perturbation, we may assume that $B$ intersects $A$ transversely finite times only near true vertices. By careful verification of orientation, we may conclude that the local contribution to the gleam near the true vertex is given as in Figure~\ref{fig:local_contr}.

The obstruction to extend $v_{\partial \overline{R}}$ on the whole $\overline{R}$ is $-\chi(\overline{R})=-\chi(R)$, 
which coincides with the self intersection number of $\overline{R}$ in $N(\Sigma_{g,n})$. 
Therefore $\gl_\oP(R)$ is given as the sum of the local contributions minus $\chi(R)$. 
\end{proof}

\begin{figure}[htbp]
\includegraphics[scale=0.75]{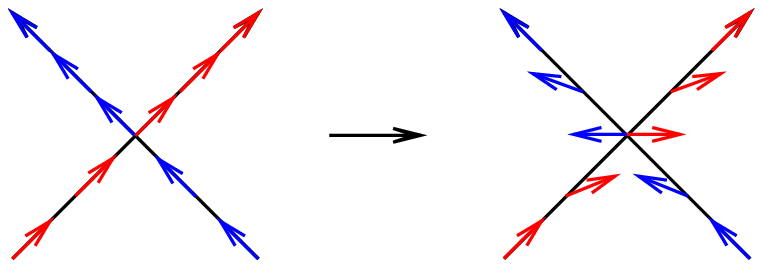}
\caption{The vector field that represents the framing of thickening.}
\label{fig:shadow_framing}
\end{figure}
\begin{figure}[htbp]
\includegraphics[scale=0.75]{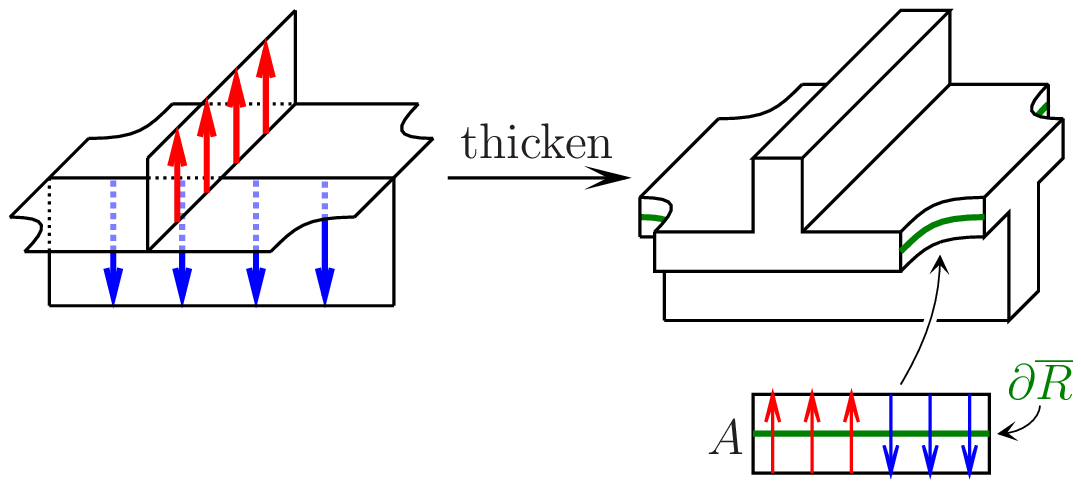}
\caption{The neighborhood of a true vertex in a slice $\Real^3$ and a part of $A$.} 
\label{fig:thickened_vertex}
\end{figure}

\begin{remark}
The internal regions $R$ 
of $X_{\oP}$ lie on $\Sigma_{g,n}$.
In the case of divides in the unit disk, in particular the case of real morsified curves,
these internal regions lie on the real plane $\Real^2\subset\Complex^2$.
\end{remark}

\begin{lemma}\label{LF-property}
The shadowed polyhedron $(X_P, \gl_P)$ of an admissible divide $P$ has an LF-structure.
\end{lemma}

\begin{proof}
The polyhedron $X_P$ is obtained from $P$ by doubling it to the divide $P_2$ and attaching annuli along $P_2$ as explained in Introduction. 
Let $X_P'$ be the polyhedron obtained from $X_P$ by 
removing the regions adjacent to $\partial \Sigma_{g,n}$ by collapsing from $\partial \Sigma_{g,n}$.
We may obtain a smooth surface $\Sigma$ from $X_P'$ by removing internal regions corresponding to the singularities of $f_P$.
Since $P$ is admissible, the inside regions of $P$ admit a checkerboard coloring with colors, say, black and write. To each edge of $P$, we assign the orientation induced from the orientation of the write region adjacent to that edge.
Two triangular regions of $P_2$ correspond to the edge and we define that the orientation of the first triangular region is positive and the second one is negative.
We can see that the orientations of all triangular regions 
are consistent, that is, $\Sigma$ is orientable.

To prove the lemma, it is enough to show that the gleam on each internal region of
$X_P'$ coincides with the one determined by the conditions~(iv) and~(v) in the definition of LF-structure.
Note that a bigon with gleam $0$ in Step~5 of the doubling (cf.~the regions labeled $c$ in Figure~\ref{fig:doubling}) is not an internal region of $X_P'$. Hence we don't need to check its gleam.
Let $R_1,\ldots,R_m$ be the internal regions of $X_P'$ corresponding to maxima, saddles and minima of $f_P$.
We order these regions such that $R_1,\ldots,R_{m_1}$ are maxima,
$R_{m_1+1},\ldots,R_{m_2}$ are saddles and  $R_{m_2+1},\ldots,R_{m}$ are minima.
The gleams of these regions are $-1$, which coincide with the condition~(v).

Now we check the coincidence of the gleams on the remaining internal regions, which are the triangular regions of $P_2$ on $\Sigma$.
There are two choices of the orientation of $\Sigma$, either the one shown on the top in Figure~\ref{fig:LP_gleam1} or the opposite one. We fix the orientation shown in the figure.

First we check the local contribution around a crossing point adjacent to a region of a maximum and a region of a minimum. As shown on the bottom in Figure~\ref{fig:LP_gleam1}, the local contribution given according to Figure~\ref{fig:local_contribute2} is $-\frac{1}{2}$, that is, the local contribution to each of the triangular regions on the top in Figure~\ref{fig:LP_gleam1} is $-\frac{1}{2}$.

\begin{figure}[htbp]
\includegraphics[scale=0.65]{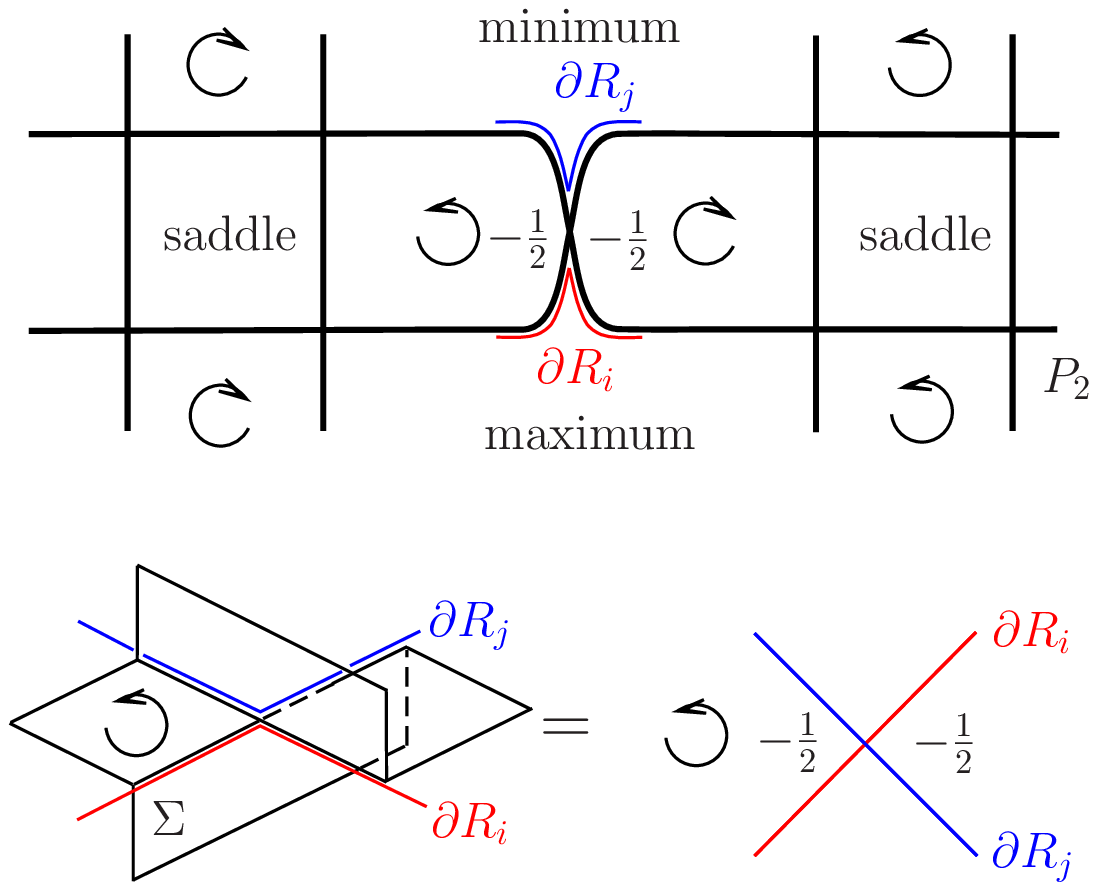}
\caption{Local contribution around a crossing point adjacent to a maximum and a minimum.}
\label{fig:LP_gleam1}
\end{figure}

Next, we check the local contribution around a crossing point adjacent to a region of a maximum and a region of a saddle,
and also around a crossing point adjacent to a region of a saddle and a region of a minimum.
As shown in Figure~\ref{fig:LP_gleam2}, the local contribution to each of the triangular regions on the top in Figure~\ref{fig:LP_gleam2} is $\frac{1}{2}$.

\begin{figure}
\includegraphics[scale=0.65]{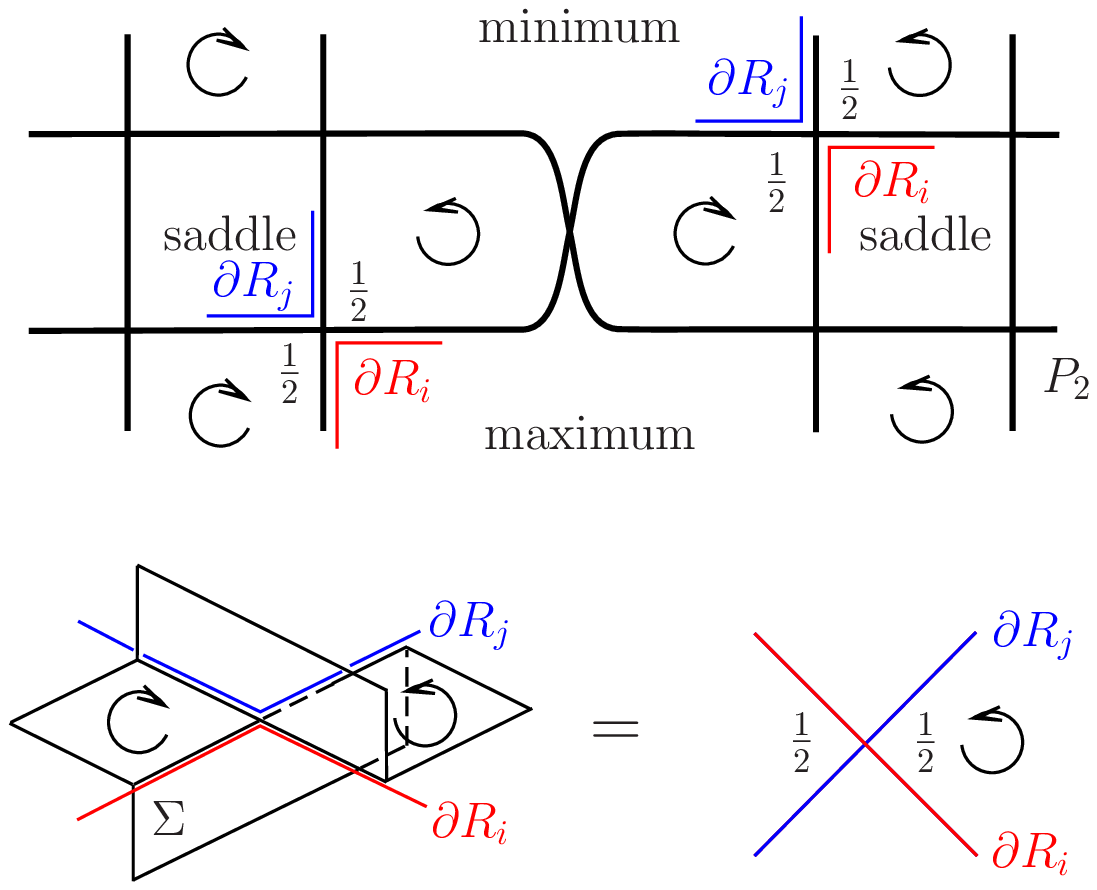}
\caption{Local contribution around crossing points adjacent to saddles.}
\label{fig:LP_gleam2}
\end{figure}

Summing up these contributions, we see that the gleam on each triangular region of $P_2$ on $\Sigma$ is $-\frac{1}{2}+\frac{1}{2}+\frac{1}{2}=\frac{1}{2}$, which coincides with the gleam $\gl_P$ of that region. This completes the proof.
\end{proof}

Let $\mathcal{X}_P$ denote the LF-structure on $(X_P,\gl_P)$ determined in the proof of Lemma~\ref{LF-property}.

\begin{proof}[Proof of Theorem~\ref{thm1}]
Let $P$ be an admissible divide on $\Sigma_{g,n}$,
$q_1, \ldots, q_\ell$ be the images of immersed intervals and circles of $P$,
$R_1,\ldots, R_m$ be the inside regions of $P$ and $c_1, \ldots, c_\delta$ be the double points of $P$.
Assign an orientation to each $q_i$.
For each point $x\in q_i$, let $I_+(x)$ (resp. $I_-(x)$) 
denote the segment in $N(\Sigma_{g,n})$
consisting of the point $x$ and the points corresponding to the
vectors tangent to $\oP$ at $x$
in the same (resp. opposite) direction to the orientation of $q_i$.
Each of $\bigcup_{x\in q_i}I_+(x)$ and $\bigcup_{x\in q_i}I_-(x)$ is an annulus one of whose boundary component lies on $\Sigma_{g,n}$ and the other lies on $\partial N(\Sigma_{g,n})$.
We denote $\bigcup_{x\in q_i}I_+(x)$ and $\bigcup_{x\in q_i}I_-(x)$ 
by $R_+(q_i)$ and $R_-(q_i)$, respectively.
The union of $\Sigma_{g,n}$, $R_+(q_i)$'s and $R_-(q_i)$'s for $i=1,\ldots,\ell$ 
is a non-simple polyhedron embedded in $N(\Sigma_{g,n})$, which we denote by $\hat X_P$.
The singular fiber $F_P^{-1}(0)\cap N(\Sigma_{g,n})$ is isotopic to 
$\bigcup_{i=1}^\ell (R_+(q_i)\cup R_-(q_i))$.
As explained in Section~\ref{sec:2.1}, $N(\Sigma_{g,n})$ is recovered from 
$F_P^{-1}(D_\eta)\cap N(\Sigma_{g,n})$ by attaching $2$-handles corresponding to the inside regions of $P$.

Next we perturb $\hat X_P$ in $N(\Sigma_{g,n})$ so that it becomes simple.
The regions $R_+(q_i)$ and $R_-(q_i)$ can be represented by vector fields based on $q_i$
as shown on the left in Figure~\ref{fig:deform_arrow},
which we denote by $v_+(q_i)$ and $v_-(q_i)$, respectively.
A deformation of $R_+(q_i)$ and $R_-(q_i)$ in $N(\Sigma_{g,n})$ can be represented by 
a deformation of the base curve $q_i$ and an isotopy of these vector fields. 
We perturb $\hat X_P$ such that the base curve becomes the doubled curve $P_2$ 
and the vector fields such that they are based on $P_2$ and 
tangent to $P_2$ in the same direction as $\oPP$, see on the right in Figure~\ref{fig:deform_arrow}.
Here the orientation of $\oPP$ is the one induced from the orientation of triangular internal regions given in the first paragraph of the proof of Lemma~\ref{LF-property}.
The obtained polyhedron $X_\oPP$ is embedded in $N(\Sigma_{g,n})$ and the embedding is represented by the gleam $\gl_\oPP$ of the oriented divide $\oPP$.
The shadowed polyhedron $(X_\oPP, \gl_\oPP)$ is nothing but $(X_P,\gl_P)$ by definition.

\begin{figure}[htbp]
\includegraphics[scale=0.75]{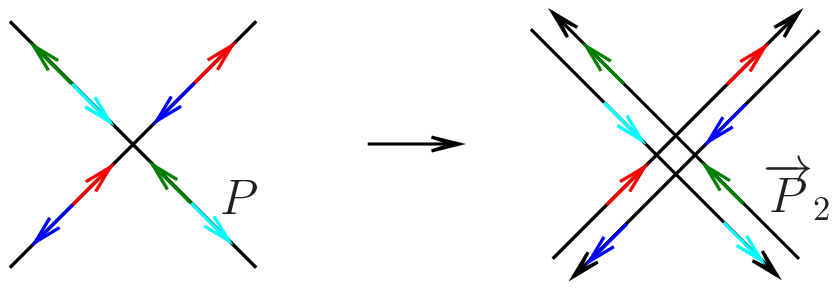}
\caption{A deformation of $R_+(q_i)$ and $R_-(q_i)$.}
\label{fig:deform_arrow}
\end{figure}

Let $\Sigma$ denote the surface
obtained from $X_P$ by removing all regions contained in the inside regions of $P$ and those containing a double point of $P$.
Since the singular fiber $F_P^{-1}(0)\cap N(\Sigma_{g,n})$ is isotopic 
to the surface $\Sigma$ outside small neighborhoods $\Nbd(c_k;N(\Sigma_{g,n}))$,
the nearby fiber $F_P^{-1}(t)\cap N(\Sigma_{g,n})$, $0<|t|\ll 1$, is also isotopic to $\Sigma$ outside $\Nbd(c_k;N(\Sigma_{g,n}))$'s.
In $\Nbd(c_k;N(\Sigma_{g,n}))$,
$F_P^{-1}(t)\cap \Nbd(c_k;N(\Sigma_{g,n}))$ and $\Sigma\cap \Nbd(c_k;N(\Sigma_{g,n}))$
are annuli in the $4$-ball $\Nbd(c_k;N(\Sigma_{g,n}))$.
Furthermore, since $F_P^{-1}(t)\cap \partial \Nbd(c_k;N(\Sigma_{g,n}))$ 
and $\Sigma\cap \partial \Nbd(c_k;N(\Sigma_{g,n}))$ are isotopic as oriented links in 
$\partial \Nbd(c_k;N(\Sigma_{g,n}))$, these annuli are isotopic in $\Nbd(c_k;N(\Sigma_{g,n}))$.
Hence $F^{-1}_P(t)\cap N(\Sigma_{g,n})$ and $\Sigma$ are isotopic.

It had been shown in Lemma~\ref{LF-property} that $(X_P, \gl_P)$ has the LF-structure $\mathcal{X}_P$.
Moreover, in the proof of Lemma~\ref{LF-property},
the order of the internal regions for
the definition of 
LF-structure is maxima, saddles and minima, which is the order of the right-handed Dehn twists of the monodromy of the fibration of the divide $P$.
Both of the Lefschetz fibrations of $\mathcal{X}_P$ on $(X_P,\gl_P)$ and $P$ are obtained from $F^{-1}_P(D_\eta)\cap N(\Sigma_{g,n})\cong \Nbd(\Sigma;N(\Sigma_{g,n}))$ by attaching $2$-handles corresponding to the inside regions of $P$ along the same vanishing cycles with fiber surface framing minus $1$ and with the same order.
Thus the two Lefschetz fibrations are isomorphic.
This completes the proof.
\end{proof}

\section{Lefschetz fibrations of certain free divides}

A free divide is a divide in the unit disk whose endpoints are not necessary on the boundary of the disk. Let $D$ denote the unit disk.

\begin{definition}
A {\it free divide} $Q$ 
is the image of a generic immersion of intervals and circles into $D$.
\end{definition}

In this paper, we only study free divides consisting of one immersed interval.
We further assume that one of the endpoints lies on $\partial D$
and the other is not adjacent to the outside region,
which is called a free divide {\it with one free endpoint}.

\begin{definition}
Let $Q$ be a free divide in $D$ consisting of one immersed interval and with one free endpoint. The {\it link} of $Q$ is defined to be the link of an oriented divide obtained from $Q$ by doubling it according to the same rule as explained in Introduction.
\end{definition}

Remark that though there are two choices for the orientation of the doubled curve $Q_2$ of $Q$,
the link-type of $L(Q)$ does not depend on this choice since they are isotopic by $\pi$-rotation of the fibers of the bundle $\partial N(D)\cap \hat B\to B\subset D$.
If one immersed interval has two free endpoints then we need to introduce signs to these endpoints to define its link, see~\cite{GI02b}. 
We also remark that we only consider a free divide not in $\Sigma_{g,n}$ but in the unit disk.
This is because we do not know the admissibility condition for free divides.

\begin{proof}[Proof of Theorem~\ref{thm2}]
Let $Q$ be a free divide in the assertion.
We first prove case~(1).
There are two edges adjacent to $c$ and the outside region, one of which is also adjacent to the region whose boundary contains the free endpoint. We denote it by $e$. 
To make a shadowed polyhedron of $Q$, we use the following doubling method:
\begin{itemize}
\item[1.] double the curve of $Q$;
\item[2.] for each endpoint of $Q$, close the corresponding two endpoints of the double curve by a small half circle;
\item[3.] for each edge of $Q$
that is neither adjacent to an endpoint nor the edge $e$,
add a crossing between the two edges of the doubled curve parallel the edge.
\end{itemize}
The doubled curve near the edge $e$ becomes as shown on the left in Figure~\ref{fig:free_endpoint} or its mirror image.
We prove the assertion in the former case. The latter case can also be proved by the same argument.

Let $\oQ$ be an oriented divide obtained from $Q$ by applying this double method,
deforming the curve near the free endpoint as shown in Figure~\ref{fig:free_endpoint}
and assigning any orientation to the doubled curve. 
Note that the link-type of the link $L(\oQ)$ of $\oQ$ does not depend on the choice of the assigned orientation.
Obviously, $L(\oQ)$ and $L(Q)$ are isotopic.
Hence it is enough to show that $L(\oQ)$ satisfies the properties in the assertion.
Let $X'_{\oQ}$ be the shadowed polyhedron obtained from the shadowed polyhedron of $\oQ$ by removing the boundary region adjacent to $\partial D$.
We may obtain the surface $\Sigma$ for an LF-structure from $X'_{\oQ}$ by removing suitable internal regions as we did for divides in the proof of Lemma~\ref{LF-property}. 
We set the orientation of $\Sigma$ as shown on the right in Figure~\ref{fig:free_endpoint}.
Regarding the vanishing cycle about the bigon in the figure as a saddle, we set the order of the regions $X'_{\oQ}\setminus \Sigma$ by the order of maxima, saddles and minima as for divides.
This order satisfies the condition~(iv) of LF-structure. Thus the assertion in case~(1) is proved.

\begin{figure}[htbp]
\includegraphics[scale=0.50]{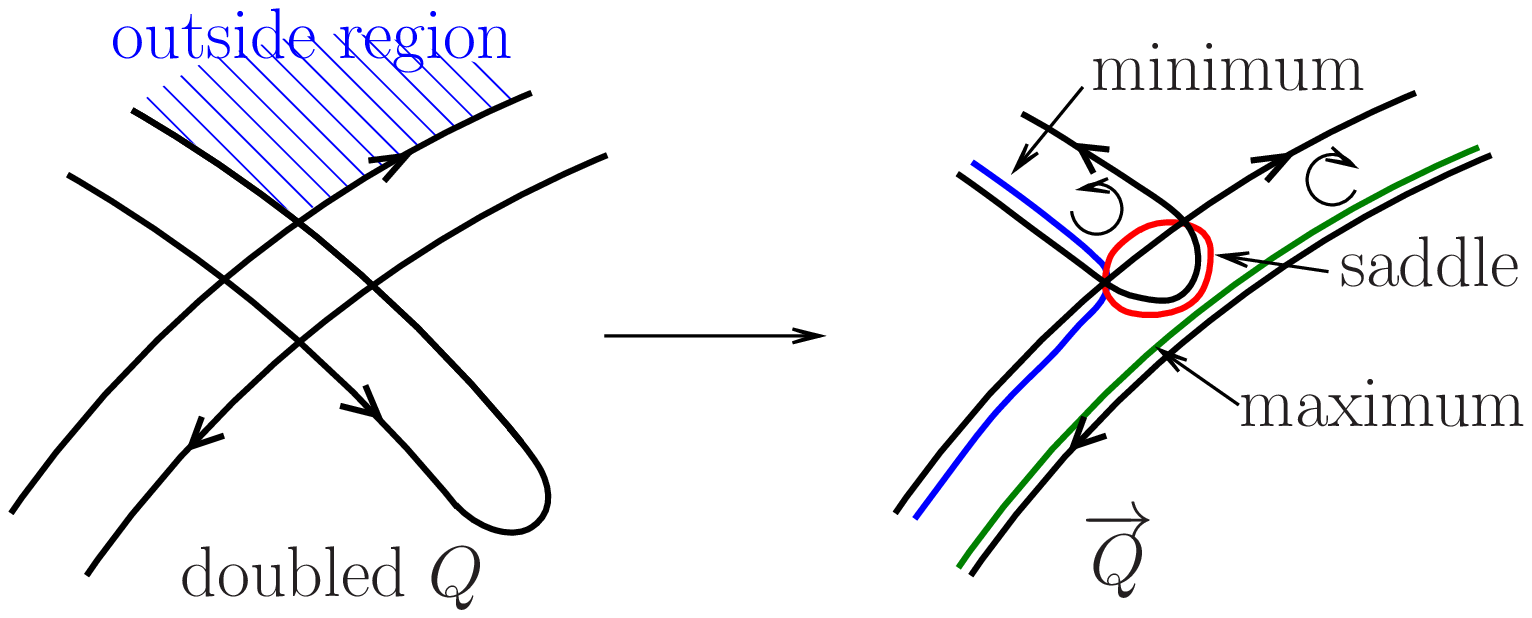}
\caption{A deformation of the doubled curve near the free endpoint in case~(1).}
\label{fig:free_endpoint}
\end{figure}

Next we prove case~(2). 
Let $c'$ be the double point of $Q$ connected to its non-free endpoint by a single edge and $e$ be the edge connecting $c$ and $c'$. 
There are two regions adjacent to $e$, one of which is adjacent to the free endpoint and we denote the other one by $R$. 
Let $e'$ be the edge of $Q$ adjacent to $c'$ and lying between $R$ and the outside region.
We apply the doubling method to $Q$ as in case~(1) with modification that,
for each of the edges $e$ and $e'$, we do not add a crossing between the two edges of the doubled curve parallel to the edge. 
The doubled curve near the edge $e$ becomes as shown on the left in Figure~\ref{fig:free_endpoint3} or its mirror image.
We prove the assertion in the former case. The latter case can also be proved by the same argument.
As in case~(1), we define $\oQ$ as in Figure~\ref{fig:free_endpoint3}, make $X'_{\oQ}$,
define $\Sigma$ and fix its orientation.
Let $C_1, C_2, C_3$ be the vanishing cycles shown on the right in Figure~\ref{fig:free_endpoint3} and $R_1,R_2,R_3$ be the internal regions of $X'_\oQ$ corresponding to these cycles. We regard $C_1$ and $C_3$ as maxima and set the order of regions $X'_\oQ\setminus\Sigma$ except $R_2$ by the order of maxima, saddles and minima as in case~(1). We then set the order of $R_2$ as $R_1<R_2<R_3$.
This order satisfies the condition~(iv) of LF-structure and the proof completes.
\end{proof}

\begin{figure}[htbp]
\includegraphics[scale=0.60]{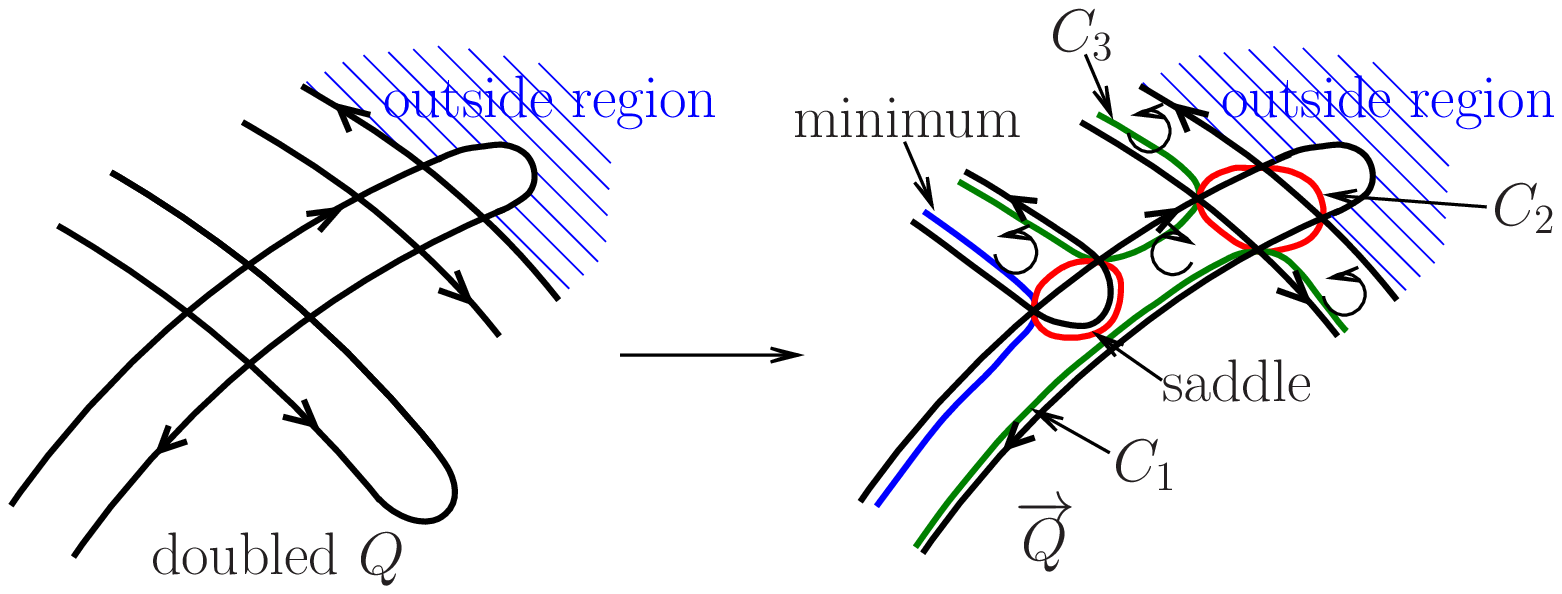}
\caption{The doubled curve around the arc connecting $c$ and the non-free endpoint in case~(2).}
\label{fig:free_endpoint3}
\end{figure}

We conclude the paper with one example.

\begin{example}
Let $Q$ be a free divide shown on the left in Figure~\ref{fig:free_divide}.
The shadowed polyhedron is given on the right.
The union of the regions with marks of orientations and the annuli attached along $\oQ$ is the surface $\Sigma$ of the LF-structure, which is a regular fiber of the Lefschetz fibration.
Let $C_1, C_2, C_3, C_4$ be the four vanishing cycles along which the regions $R_1, R_2, R_3, R_4$ with gleam $-1$ are attached.
The monodromy of the fibered link $L(Q)$ is the product of right-handed Dehn twists $\sigma_1, \sigma_2, \sigma_3, \sigma_4$ along $C_1, C_2, C_3, C_4$ in this order. The monodromy matrix is given as
\[
\begin{split}
M_{\sigma_4}M_{\sigma_3}M_{\sigma_2}M_{\sigma_1}
&=
\begin{pmatrix} 1 & 0 & 0 & 0 \\ 0 & 1 & 0 & 0 \\ 0 & 0 & 1 & 0 \\ -1 & 1 & 1 & 1 \end{pmatrix}
\begin{pmatrix} 1 & 0 & 0 & 0 \\ 0 & 1 & 0 & 0 \\ 0 & 0 & 1 & -1 \\ 0 & 0 & 0 & 1 \end{pmatrix}
\begin{pmatrix} 1 & 0 & 0 & 0 \\ 1 & 1 & 0 & -1 \\ 0 & 0 & 1 & 0 \\ 0 & 0 & 0 & 1 \end{pmatrix}
\begin{pmatrix} 1 & -1 & 0 & 1 \\ 0 & 1 & 0 & 0 \\ 0 & 0 & 1 & 0 \\ 0 & 0 & 0 & 1 \end{pmatrix} \\
&=
\begin{pmatrix} 1 & -1 & 0 & 1 \\ 1 & 0 & 0 & 0 \\ 0 & 0 & 1 & -1 \\ 0 & 1 & 1 & -1 \end{pmatrix},
\end{split}
\]
where $M_{\sigma_i}$ is the monodromy matrix of $\sigma_i$.
The characteristic polynomial of this matrix is $t^4-t^3+t^2-t+1$, which is the Alexander polynomial of the $(2,5)$-torus knot.
It is known in~\cite{GI02b} that the link of this free divide is a $(2,5)$-torus knot.
Actually, we can check that it is a positive $(2,5)$-torus knot.
We can also check it by describing a Kirby diagram of the shadowed polyhedron.
\end{example}

\begin{figure}[htbp]
\includegraphics[scale=0.65]{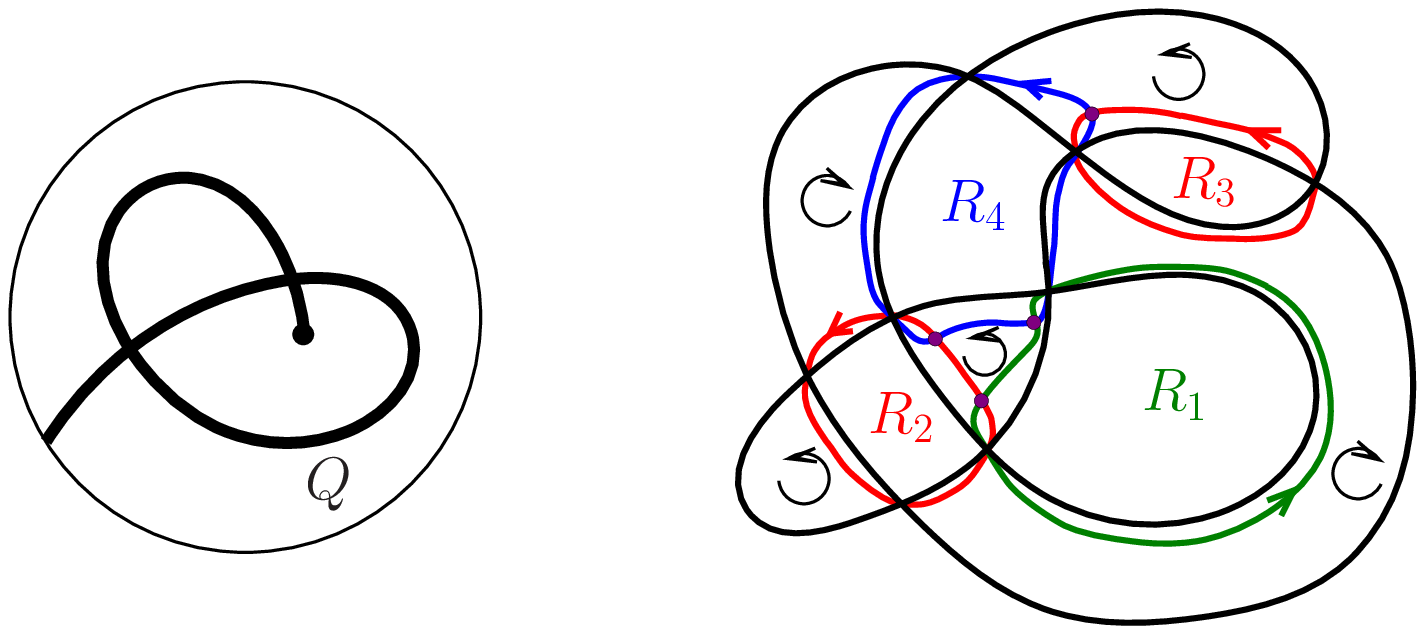}
\caption{An example of a free divide and its shadowed polyhedron.}
\label{fig:free_divide}
\end{figure}

\end{document}